\documentclass[a4paper,oneside,11pt]{article}%
\usepackage{makeidx}
\usepackage[english]{babel}
\usepackage{amsmath}
\usepackage{amsfonts}
\usepackage{amssymb}
\usepackage{stmaryrd}
\usepackage{graphicx}
\usepackage{mathrsfs}
\usepackage[colorlinks,linkcolor=red,anchorcolor=blue,citecolor=blue,urlcolor=blue]{hyperref}
\usepackage[symbol*,ragged]{footmisc}

\allowdisplaybreaks[4]
\providecommand{\U}[1]{\protect\rule{.1in}{.1in}}
\providecommand{\U}[1]{\protect \rule{.1in}{.1in}}

\newtheorem{theorem}{Theorem}[section]

\newtheorem{lemma}[theorem]{Lemma}

\newenvironment{proof}[1][Proof]{\noindent \textbf{#1.} }{\  \rule{0.5em}{0.5em}}
\numberwithin{equation}{section}

\begin{document}

\title{On the least almost--prime in arithmetic progression }
%  \author{Min Zhang\footnotemark   \vspace*{-5mm} \\
   %  \small Department of Mathematics, China University of Mining and Technology \vspace*{-5mm} \\
  %  \small  Beijing 100083, P. R. China  }

\author{Jinjiang Li\footnotemark[1] \,\,\,\,\,  \& \,\,\, Min Zhang\footnotemark[2]\,\,\,\,\,\,\,  \& \,\,\,
        Yingchun Cai\footnotemark[3]
                    \vspace*{-4mm} \\
     $\textrm{\small Department of Mathematics, China University of Mining and Technology\footnotemark[1]}$
                    \vspace*{-4mm} \\
     \small  Beijing 100083, P. R. China
                     \vspace*{-4mm}  \\
     $\textrm{\small School of Applied Science, Beijing Information Science and Technology University\footnotemark[2]}$
                    \vspace*{-4mm}  \\
     \small  Beijing 100192, P. R. China
                    \vspace*{-4mm}  \\
     $\textrm{\small School of Mathematical Science, Tongji University\footnotemark[3]}$
                    \vspace*{-4mm}  \\
     \small  Shanghai 200092, P. R. China  }
                   %\vspace*{-4mm}}

\footnotetext[2]{Corresponding author. \\
    \quad\,\, \textit{ E-mail addresses}:
     \href{mailto:jinjiang.li.math@gmail.com}{jinjiang.li.math@gmail.com} (J. Li),
     \href{mailto:min.zhang.math@gmail.com}{min.zhang.math@gmail.com} (M. Zhang),\\
     \qquad \qquad\qquad\quad\quad\quad \,\,\,\,\,
     \href{mailto:yingchuncai@tongji.edu.cn}{yingchuncai@tongji.edu.cn} (Y. Cai).  }

\date{}
\maketitle

{\textbf{Abstract}}: Let $\mathcal{P}_r$ denote an almost--prime with at most $r$ prime factors, counted
according to multiplicity. Suppose that $a$ and $q$ are positive integers satisfying $(a,q)=1$. Denote by $\mathcal{P}_2(a,q)$ the least almost--prime $\mathcal{P}_2$ which satisfies $\mathcal{P}_2\equiv a\pmod q$. In this paper, it is proved that for sufficiently large $q$, there holds
\begin{equation*}
   \mathcal{P}_2(a,q)\ll q^{1.8345}.
\end{equation*}
This result constitutes an improvement upon that of Iwaniec \cite{Iwaniec-1982}, who obtained the same conclusion, but for the range $1.845$ in place of $1.8345$.


{\textbf{MR(2020) Subject Classification}}: 11N13, 11N35, 11N36

\section{Introduction and main result}
Let $\mathcal{P}_r$ denote an almost--prime with at most $r$ prime factors, counted according to multiplicity.
In this paper, we shall investigate the occurrence of almost--primes in arithmetic progressions. This problem correspond to a well--known conjecture concerning prime numbers. The conjecture states that, if $(a,q)=1$,
there exists a prime $p$ satisfying
\begin{equation}\label{assumption}
   p\equiv a\!\!\!\pmod q,\quad  p\leqslant q^2\quad (q\geqslant2).
\end{equation}
Indeed the bound for $p$ may presumably be reduced to $p\ll q(\log q)^2$. Unfortunately, we cannot prove
(\ref{assumption}) even on the assumption of the generalized Riemann hypothesis. The nearest approach seems
to be the conditional estimate $p\ll (\varphi(q))^2(\log q)^4$, which follows from Theorem 6 of Titchmarsh \cite{Titchmarsh-1930}. However, as an approach to approximate this conjecture, we can consider almost--primes
in arithmetic progression. Many authors investigated this approximation in the past time. Denote by $\mathcal{P}_2(a,q)$ the least almost--prime $\mathcal{P}_2$ which satisfies $\mathcal{P}_2\equiv a\pmod q$. In 1965, Levin \cite{Levin-1965} showed that $\mathcal{P}_2(a,q)\ll q^{2.3696}$. Later, Richert pointed out that by using the method in \cite{Jurkat-Richert-1965}, the exponent can be replaced by $\frac{25}{11}+\varepsilon$. Afterwards, Halberstam and Richert gave the result that the exponent can be replaced by $\frac{11}{5}$ in
their monograph \cite{Halberstam-Richert-book}, Chapter 9. Motohashi \cite{Motohashi-1976}, in 1976, gave the exponent $2+\varepsilon$ subject to a certain unproved hypothesis. In 1978, Heath--Brown first gave an unconditional bound, stronger than (\ref{assumption}), for almost--primes $\mathcal{P}_2$. He showed that
$\mathcal{P}_2(a,q)\ll q^{1.965}$. After that, in 1982, Iwaniec \cite{Iwaniec-1982} improved Heath--Brown's result
and derive that $\mathcal{P}_2(a,q)\ll q^{1.845}$.

In this paper, we shall continue to improve the result of Iwaniec \cite{Iwaniec-1982} and establish
the following theorem.

\begin{theorem}\label{Theorem}
  Suppose that $a$ and $q$ are positive integers satisfying $(a,q)=1$. Let $\mathcal{P}_2(a,q)$ be the least almost--prime $\mathcal{P}_2$ which satisfies $\mathcal{P}_2\equiv a\pmod q$. Then for sufficiently large $q$, there holds
\begin{equation*}
   \mathcal{P}_2(a,q)\ll q^{1.8345}.
\end{equation*}
\end{theorem}

\noindent
\textbf{Remark.} Our improvement comes from using distinct methods to deal with the different parts of the sifting sum with more delicate techniques, combining with linear sieve results of Iwaniec \cite{Iwaniec-1980} with bilinear forms for the remainder term and the two--dimensional sieve of Selberg.

\section{Notation and Preliminaries}

Throughout this paper, we always denote primes by $p$. $\varepsilon$ always denotes an arbitrarily small
positive constant, which may not be the same at different occurrences. As usual, we use
$\varphi(n),\mu(n),\tau(n)$ to denote Euler's function, M\"{o}bius' function, and Dirichlet divisor function, respectively.
Moreover, $\Omega(n)$ denotes the number of prime factors of $n$, counted according to multiplicity.
Let $(m_1,m_2,\dots,m_k)$ and $[m_1,m_2,\dots,m_k]$ be the greatest common divisor and the least common multiple of $m_1,m_2,\dots,m_k$, respectively. Also, $f(x)\ll g(x)$ means that $f(x)=O(g(x))$. $\mathcal{P}_r$ always denotes an almost--prime with at most $r$ prime factors, counted according to multiplicity.

Let $\mathscr{A}$ be a finite sequence of integers, and $\mathscr{P}$ a set of primes. For given $z\geqslant2$, we denote
\begin{equation*}
  P(z)=\prod_{\substack{p<z \\p\in\mathscr{P}}}p.
\end{equation*}
Define the sifting function as
\begin{equation*}
  S(\mathscr{A},\mathscr{P},z)=\big|\{a\in\mathscr{A}:(a,P(z))=1 \}\big|.
\end{equation*}
For $d|P(z)$, define $\mathscr{A}_d=\{a\in\mathscr{A}:a\equiv0\pmod d\}$. Moreover, we assume
that $|\mathscr{A}_d|$ may be written in the form
\begin{equation*}
  \big|\mathscr{A}_d\big|=\frac{\omega(d)}{d}X+r(\mathscr{A},d),
\end{equation*}
where $\omega(d)$ is a multiplicative function satisfying $0<\omega(p)<p$ for $p\in\mathscr{P}$; $X$ is an approximation to
$|\mathscr{A}|$ independent of $d$. In addition, $\omega(d)d^{-1}X$ is regarded as a main term of $|\mathscr{A}_d|$;
$r(\mathscr{A},d)$ is regarded as an error term of $|\mathscr{A}_d|$, which is expected to be small on average over $d$.
Also, we assume that the function $\omega(p)$ is constant on average over $p$ in $\mathscr{P}$, which means that
\begin{equation*}
  \prod_{\substack{z_1\leqslant p<z_2\\ p\in\mathscr{P}}}\bigg(1-\frac{\omega(p)}{p}\bigg)^{-1}
  \leqslant\frac{\log z_2}{\log z_1}\bigg(1+\frac{K}{\log z_1}\bigg)
\end{equation*}
for all $z_2>z_1\geqslant2$, where $K$ is a constant satisfying $K\geqslant1$.

\begin{lemma}\label{f(u)}
Let $F(u)$ and $f(u)$ be continuous functions, which satisfy the following differential--difference equations
\begin{equation*}
  \begin{cases}
    F(u)=\displaystyle\frac{2e^\gamma}{u},\quad f(u)=0,\quad\textrm{for}\quad 1\leqslant u\leqslant2,\\
    (uF(u))'=f(u-1),\quad (uf(u))'=F(u-1), \quad\textrm{for}\quad u\geqslant2.
  \end{cases}
\end{equation*}
Then we have
\begin{align*}
F(u)= & \,\,\frac{2e^{\gamma}}{u}, \qquad \textrm{for}\qquad 0< u\leqslant3,
               \nonumber \\
F(u)= & \,\,\frac{2e^{\gamma}}{u}\bigg(1+\int_2^{u-1}\frac{\log(t-1)}{t}\mathrm{d}t\bigg),
             \qquad \textrm{for}\qquad 3\leqslant u\leqslant5,
               \nonumber \\
 f(u)= & \,\, \frac{2e^\gamma}{u}\bigg(\log(u-1)+\int_3^{u-1}\frac{\mathrm{d}t_1}{t_1}
 \int_{2}^{t_1-1}\frac{\log(t_2-1)}{t_2}\mathrm{d}t_2\bigg), \qquad \textrm{for \,\,\,$4\leqslant u\leqslant 6$}.
\end{align*}
\end{lemma}
\begin{proof}
 See (2.8) of Chapter 8 in Halberstam and Richert \cite{Halberstam-Richert-book}, and pp. 126--127 of Pan and Pan \cite{Pan-Pan-book}.  $\hfill$
\end{proof}

\section{Proof of Theorem \ref{Theorem}}

Let $\mathscr{A}=\{n: n\leqslant x, n\equiv a\pmod q\}$, where $(a,q)=1,\,x^{1/2}<q\leqslant x^{3/5}$. Set
\begin{equation*}
\mathscr{P}=\{p:\,p \nmid q\},\qquad M=x^{1-3\varepsilon}q^{-1},\qquad N=x^{\frac{1}{2}-4\varepsilon}q^{-\frac{3}{4}},\qquad D=MN.
\end{equation*}
Moreover, we put
\begin{equation*}
 \delta=0.86,\qquad \theta=1.8345, \qquad x=q^\theta, \qquad y=q^\delta.
\end{equation*}
We write $S(\mathscr{A},z)$ as abbreviation of $S(\mathscr{A},\mathscr{P},z)$ for convenience. For the notations defined as above, we have $M<y<D$ and consider the weighted sum with Richert's weights of logarithmic type
\begin{equation*}
W(\mathscr{A};z,y)=\sum_{\substack{n\in\mathscr{A}\\ (n,P(z))=1}}
\Bigg(1-\frac{1}{\lambda}\sum_{\substack{z\leqslant p<y\\ p|n}}\bigg(1-\frac{\log p}{\log y}\bigg)\Bigg),
\end{equation*}
where $z=D^{5/23},\,\lambda=3-\displaystyle\frac{\log x}{\log y}-\varepsilon$. For convenience, we write
\begin{equation*}
  \mathcal{W}(n)=1-\frac{1}{\lambda}\sum_{\substack{z\leqslant p<y\\ p|n}}\bigg(1-\frac{\log p}{\log y}\bigg).
\end{equation*}
Then we have
\begin{equation}\label{W-fj}
  W(\mathscr{A};z,y)=\sum_{\substack{n\in\mathscr{A}\\ (n,P(z))=1\\ \Omega(n)\leqslant2}} \mathcal{W}(n)
                    +\sum_{\substack{n\in\mathscr{A}\\ (n,P(z))=1\\ \Omega(n)\geqslant3\\ \mu(n)\not=0}} \mathcal{W}(n)
                    +\sum_{\substack{n\in\mathscr{A}\\ (n,P(z))=1\\ \Omega(n)\geqslant3\\ \mu(n)=0}} \mathcal{W}(n).
\end{equation}
Obviously, we have
\begin{align}
            \sum_{\substack{n\in\mathscr{A}\\ (n,P(z))=1\\ \Omega(n)\geqslant3\\ \mu(n)=0}} \mathcal{W}(n)
 \ll & \,\, \sum_{\substack{n\in\mathscr{A}\\ (n,P(z))=1\\ \mu(n)=0}}\tau(n)
            \ll x^\varepsilon\sum_{\substack{z\leqslant p\leqslant x^{1/2}\\ p\nmid q}}
            \sum_{\substack{k\leqslant(x-a_1)/q\\ k\equiv -a_1\overline{q} \!\!\!\!\!\pmod{p^2}}} 1
                  \nonumber \\
 \ll & \,\,  x^\varepsilon\sum_{z\leqslant p\leqslant x^{1/2}}\bigg(\frac{x}{p^2q}+1\bigg)
            \ll x^\varepsilon\bigg(\frac{x}{qz}+x^{1/2}\bigg)=o\bigg(\frac{x^{1-\varepsilon}}{\varphi(q)}\bigg),
\end{align}
where the integer $a_1$ satisfies $1\leqslant a_1\leqslant q$ and $a_1\equiv a\!\!\pmod q$, while the integer $\overline{q}$ satisfies $q\overline{q}\equiv1\!\!\pmod{p^2}$ with respect to the prime $p$ satisfying
$z\leqslant p\leqslant x^{1/2}$ and $p\nmid q$.

For a given integer $n$ with $n\leqslant x,\, n\equiv a\!\! \pmod q,\, (a,q)=1,\, (n,P(z))=1$ and $\mu(n)\not=0$, the weight
$\mathcal{W}(n)$ in the sum $W(\mathscr{A};z,y)$ satisfies
\begin{align}\label{weight-upper}
 1-\frac{1}{\lambda}\sum_{\substack{z\leqslant p<y\\ p|n}}\bigg(1-\frac{\log p}{\log y}\bigg)
 \leqslant & \,\, \frac{1}{\lambda}\Bigg(\lambda-\sum_{p|n}\bigg(1-\frac{\log p}{\log y}\bigg)\Bigg)
                     \nonumber  \\
 = & \,\, \frac{1}{\lambda}\bigg(3-\frac{\log x}{\log y}-\varepsilon-\Omega(n)+\frac{\log n}{\log y}\bigg)
          <\frac{1}{\lambda}\big(3-\Omega(n)\big),
\end{align}
and thus $\mathcal{W}(n)<0$ for $\Omega(n)\geqslant3$. From (\ref{W-fj})--(\ref{weight-upper}), we know that
\begin{align}\label{aim-lower}
          \sum_{\substack{n\in\mathscr{A}\\ (n,P(z))=1\\ \Omega(n)\leqslant 2}}\mathcal{W}(n)
 = & \,\, W(\mathscr{A};z,y)-\sum_{\substack{n\in\mathscr{A}\\ (n,P(z))=1\\ \Omega(n)\geqslant 3\\ \mu(n)\not=0}}\mathcal{W}(n)
          +o\bigg(\frac{x^{1-\varepsilon}}{\varphi(q)}\bigg)
                 \nonumber \\
 \geqslant & \,\, W(\mathscr{A};z,y)+o\bigg(\frac{x^{1-\varepsilon}}{\varphi(q)}\bigg).
\end{align}
For $W(\mathscr{A};z,y)$, we have
\begin{align}\label{W(A;z,y)-expan}
           W(\mathscr{A};z,y)
 = & \,\, \sum_{\substack{n\in\mathscr{A}\\ (n,P(z))=1}}1-\frac{1}{\lambda}
          \sum_{\substack{z\leqslant p<y\\ p\nmid q}}\bigg(1-\frac{\log p}{\log y}\bigg)
          \sum_{\substack{n\in\mathscr{A}\\ n\equiv0 \!\!\!\!\!\pmod p\\ (n,P(z))=1}}1
                    \nonumber \\
 = & \,\, S(\mathscr{A},z)-\frac{1}{\lambda}
          \sum_{\substack{z\leqslant p<y\\ p\nmid q}}\bigg(1-\frac{\log p}{\log y}\bigg)S(\mathscr{A}_p,z).
\end{align}
We appeal to Theorem 1 of Iwaniec \cite{Iwaniec-1980} for linear sieve results with bilinear forms for the remainder
term, which simply gives
\begin{equation}\label{S(A,z)-lower-1}
S(\mathscr{A},z)\geqslant \frac{x}{\varphi(q)}V(z)\bigg(f\bigg(\frac{23}{5}\bigg)+O\big((\log D)^{-1/3}\big)\bigg)-R^-,
\end{equation}
where
\begin{equation*}
R^-=\sum_{\ell<\exp(8\varepsilon^{-3})}\mathop{\sum_{m<M}\sum_{n<N}}_{(mn,q)=1}
        a_m^-(\ell)b_n^-(\ell)r(\mathscr{A},mn),
\end{equation*}
and
\begin{equation}\label{V(z)-Mertens-asy}
V(z)=\prod_{p<z}\bigg(1-\frac{1}{p}\bigg)=\frac{e^{-\gamma}}{\log z}\bigg(1+O\bigg(\frac{1}{\log z}\bigg)\bigg)
\end{equation}
by the Mertens' prime number theorem (See \cite{Mertens-1874}). By Theorem 5 of Iwaniec \cite{Iwaniec-1982},
one has
\begin{equation}\label{S(A,z)-error-upper}
R^-\ll\frac{x^{1-\varepsilon}}{\varphi(q)}.
\end{equation}
From Lemma \ref{f(u)}, (\ref{S(A,z)-lower-1}), (\ref{V(z)-Mertens-asy}) and (\ref{S(A,z)-error-upper}), we obtain
\begin{equation}\label{S(A,z)-lower}
S(\mathscr{A},z)\geqslant\frac{x}{\varphi(q)\log D}
\bigg\{2\bigg(\log \frac{18}{5}
+\int_3^\frac{18}{5}\frac{\mathrm{d}t_1}{t_1}\int_2^{t_1-1}\frac{\log(t_2-1)}{t_2}\mathrm{d}t_2\bigg)\bigg\}(1+O(\varepsilon)).
\end{equation}
For the second term in (\ref{W(A;z,y)-expan}), we write it into three parts
\begin{align}\label{W-second-sum-divided}
    & \,\, \sum_{\substack{z\leqslant p<y\\ p\nmid q}}\bigg(1-\frac{\log p}{\log y}\bigg)S(\mathscr{A}_p,z)
                    \nonumber \\
= & \,\, \sum_{\substack{z\leqslant p<D^{8/23}\\ p\nmid q}}\bigg(1-\frac{\log p}{\log y}\bigg)S(\mathscr{A}_p,z)
         +\sum_{\substack{D^{8/23}\leqslant p<M\\ p\nmid q}}\bigg(1-\frac{\log p}{\log y}\bigg)S(\mathscr{A}_p,z)
                     \nonumber \\
  & \,\,  +\sum_{\substack{M\leqslant p<y\\ p\nmid q}}\bigg(1-\frac{\log p}{\log y}\bigg)S(\mathscr{A}_p,z).
\end{align}
Henceforth, we shall use two distinct methods to deal with the sums in (\ref{W-second-sum-divided}). For the first and the second sum in (\ref{W-second-sum-divided}), we shall appeal to linear sieve results of Iwaniec with bilinear forms for the remainder term. On the other hand, we will treat the third sum by the two--dimensional sieve of Selberg.

Now, we deal with the first sum in (\ref{W-second-sum-divided}). For each $S(\mathscr{A}_p,z)$, by Theorem 1 of Iwaniec \cite{Iwaniec-1980}, we derive that 
\begin{align*}
                  S(\mathscr{A}_p,z)
 \leqslant & \,\, \frac{x}{p\varphi(q)}V(z)\Bigg(F\bigg(\frac{\log(D/p)}{\log z}\bigg)+O\big(\log^{-1/3}D\big)\Bigg)
                      \nonumber \\
           & \,\,  +\sum_{\ell<\exp(8\varepsilon^{-3})}\mathop{\sum_{m<M/p}\sum_{n<N}}_{(mn,q)=1}
                   a_m^+(\ell)b_n^+(\ell)r(\mathscr{A},pmn)
                       \nonumber \\
      = & \,\, \frac{x(2+O(\varepsilon))}{p\varphi(q)\log(D/p)}\Bigg(1+\int_2^{\frac{\log(D/p)}{\log z}-1}
               \frac{\log(t-1)}{t}\mathrm{d}t\Bigg)
                        \nonumber \\
        & \,\, +\sum_{\ell<\exp(8\varepsilon^{-3})}\mathop{\sum_{m<M/p}\sum_{n<N}}_{(mn,q)=1}
                   a_m^+(\ell)b_n^+(\ell)r(\mathscr{A},pmn),
\end{align*}
where $|a_m^+(\ell)|\leqslant1,\,|b_n^+(\ell)|\leqslant1$. Summing over $p\in[z,D^{8/23}),p\nmid q$ with an interpretation
that $pm$ as one variable of the summation while $n$ as the other, then, according to Theorem 5 of Iwaniec
\cite{Iwaniec-1982}, the final remainder term arising is $\ll x^{1-\varepsilon}/\varphi(q)$. Therefore, we deduce that
\begin{align}\label{sum_3-1}
   &\,\, \sum_{\substack{z\leqslant p<D^{8/23}\\ p\nmid q}}\bigg(1-\frac{\log p}{\log y}\bigg)S(\mathscr{A}_p,z)
                       \nonumber \\
\leqslant &\,\, \sum_{z\leqslant p<D^{8/23}}\frac{\log(y/p)}{\log y}\cdot\frac{x(2+O(\varepsilon))}{p\varphi(q)\log(D/p)}
                       \nonumber \\
          &\,\, \qquad \times\Bigg(1+\int_2^{\frac{\log(D/p)}{\log z}-1}\frac{\log(t-1)}{t}\mathrm{d}t\Bigg)
                +O\bigg(\frac{x^{1-\varepsilon}}{\varphi(q)}\bigg)
                       \nonumber \\
  = &\,\, \frac{x(2+O(\varepsilon))}{\varphi(q)\log D}\cdot\frac{\log D}{\log y}\sum_{z\leqslant p<D^{8/23}}
           \frac{\log(y/p)}{p\log(D/p)}
                       \nonumber \\
    &\,\, \qquad\times \Bigg(1+\int_2^{\frac{\log(D/p)}{\log z}-1}\frac{\log(t-1)}{t}\mathrm{d}t\Bigg)
           +O\bigg(\frac{x^{1-\varepsilon}}{\varphi(q)}\bigg).
\end{align}
By partial summation and by prime number theorem, it is easy to derive that
\begin{align}\label{sum_3-coeff}
   & \,\, \frac{\log D}{\log y}\sum_{z\leqslant p<D^{8/23}}\frac{\log(y/p)}{p\log(D/p)}
          \Bigg(1+\int_2^{\frac{\log(D/p)}{\log z}-1}\frac{\log(t-1)}{t}\mathrm{d}t\Bigg)
                 \nonumber \\
 = & \,\, \frac{6\theta-7}{4\delta}\int_{\frac{30\theta-35}{92}}^{\frac{12\theta-14}{23}}
          \frac{\delta-\beta}{\beta(\frac{3\theta}{2}-\frac{7}{4}-\beta)}
          \Bigg(1+\int_2^{\frac{108\theta-126-92\beta}{30\theta-35}}
          \frac{\log(t-1)}{t}\mathrm{d}t\Bigg)\mathrm{d}\beta+O(\varepsilon).
\end{align}
Next, we shall deal with the second sum in (\ref{W-second-sum-divided}), which is similar to the first sum. For each 
$S(\mathscr{A}_p,z)$, by Theorem 1 of Iwaniec \cite{Iwaniec-1980}, we get
\begin{align*}
                  S(\mathscr{A}_p,z)
 \leqslant & \,\, \frac{x}{p\varphi(q)}V(z)\Bigg(F\bigg(\frac{\log(D/p)}{\log z}\bigg)+O\big(\log^{-1/3}D\big)\Bigg)
                      \nonumber \\
           & \,\,  +\sum_{\ell<\exp(8\varepsilon^{-3})}\mathop{\sum_{m<M/p}\sum_{n<N}}_{(mn,q)=1}
                   a_m^+(\ell)b_n^+(\ell)r(\mathscr{A},pmn)
                       \nonumber \\
      = & \,\, \frac{x(2+O(\varepsilon))}{p\varphi(q)\log(D/p)}
               +\sum_{\ell<\exp(8\varepsilon^{-3})}\mathop{\sum_{m<M/p}\sum_{n<N}}_{(mn,q)=1}
                   a_m^+(\ell)b_n^+(\ell)r(\mathscr{A},pmn),
\end{align*}
where $|a_m^+(\ell)|\leqslant1,\,|b_n^+(\ell)|\leqslant1$. Summing over $p\in[D^{8/23},M),p\nmid q$ with an interpretation
that $pm$ as one variable of the summation while $n$ as the other, then, according to Theorem 5 of Iwaniec
\cite{Iwaniec-1982}, the final remainder term arising is $\ll x^{1-\varepsilon}/\varphi(q)$. Hence one get
\begin{align}\label{sum_1-1}
   &\,\, \sum_{\substack{D^{8/23}\leqslant p<M\\ p\nmid q}}\bigg(1-\frac{\log p}{\log y}\bigg)S(\mathscr{A}_p,z)
                       \nonumber \\
\leqslant &\,\, \sum_{D^{8/23}\leqslant p<M}\frac{\log(y/p)}{\log y}\cdot\frac{x(2+O(\varepsilon))}{p\varphi(q)\log(D/p)}
                +O\bigg(\frac{x^{1-\varepsilon}}{\varphi(q)}\bigg)
                       \nonumber \\
  = &\,\, \frac{x(2+O(\varepsilon))}{\varphi(q)\log D}\cdot\frac{\log D}{\log y}\sum_{D^{8/23}\leqslant p<M}
           \frac{\log(y/p)}{p\log(D/p)}+O\bigg(\frac{x^{1-\varepsilon}}{\varphi(q)}\bigg).
\end{align}
By partial summation and by prime number theorem, one gets
\begin{equation}\label{sum_1-coeff}
\frac{\log D}{\log y}\sum_{D^{8/23}\leqslant p<M}\frac{\log(y/p)}{p\log(D/p)}
=\frac{6\theta-7}{4\delta}
\int_{\frac{12\theta-14}{23}}^{\theta-1}
\frac{\delta-\beta}{\beta(\frac{3\theta}{2}-\frac{7}{4}-\beta)}\mathrm{d}\beta+O(\varepsilon).
\end{equation}
Finally, we shall deal with the third sum, which appears in (\ref{W-second-sum-divided}), in a different manner without appealing to Theorem 1 of Iwaniec \cite{Iwaniec-1980}. We begin with ignoring the fact that $p$ is a prime
and obtaining
\begin{align}\label{sum_2-upper-1}
     \sum_{\substack{M\leqslant p<y\\ p\nmid q}}\bigg(1-\frac{\log p}{\log y}\bigg)S(\mathscr{A}_p,z)
 = & \,\, \sum_{\substack{M\leqslant p<y\\ (p,q)=1}}\bigg(1-\frac{\log p}{\log y}\bigg)
          \sum_{\substack{m\in\mathscr{A}\\ m\equiv0\!\!\!\!\!\pmod p\\ (m,P(z))=1}}1
                       \nonumber \\
 \leqslant & \,\, \sum_{M\leqslant n< y}\bigg(1-\frac{\log n}{\log y}\bigg)
                  \sum_{\substack{m\in\mathscr{A}\\ m\equiv0\!\!\!\!\!\pmod n \\ (m,P(z))=1}}1,
\end{align}
where $n$ runs over all integers in the interval $[M,y)$. Let $\{\lambda^+(d)\}$ be an upper bound sieve of level $D_1$, i.e. a sequence of real numbers satisfying
\begin{equation*}
|\lambda^+(d)|\leqslant1,\qquad \lambda^+(d)=0\qquad \textrm{for}\quad d\geqslant D_1 \quad
 \textrm{or}\quad\mu(d)=0,
\end{equation*}
and
\begin{equation*}
\sum_{d|n}\mu(d)\leqslant\sum_{d|n}\lambda^+(d).
\end{equation*}
Then we get
\begin{align}\label{sum_2-upper-2}
   & \,\, \sum_{M\leqslant n< y}\bigg(1-\frac{\log n}{\log y}\bigg)
          \sum_{\substack{m\in\mathscr{A}\\ m\equiv0\!\!\!\!\!\pmod n\\ (m,P(z))=1}}1
                    \nonumber \\
   = & \,\,\sum_{M\leqslant n< y}\bigg(1-\frac{\log n}{\log y}\bigg)
           \sum_{\substack{m\in\mathscr{A}\\ m\equiv0\!\!\!\!\!\pmod n}}\sum_{d\mid (m,P(z))}\mu(d)
                       \nonumber \\
\leqslant & \,\, \sum_{M\leqslant n< y}\bigg(1-\frac{\log n}{\log y}\bigg)
                  \sum_{\substack{m\in\mathscr{A}\\ m\equiv0\!\!\!\!\!\pmod n}}\sum_{d\mid (m,P(z))}\lambda^+(d)
                       \nonumber \\
   = & \,\, \sum_{\substack{d<D_1\\ d|P(z)}}\lambda^+(d)
            \sum_{M\leqslant n< y}\bigg(1-\frac{\log n}{\log y}\bigg)
            \sum_{\substack{m\in\mathscr{A}\\ m\equiv0\!\!\!\!\!\pmod {[d,n]}}}1
                      \nonumber \\
= & \,\, \sum_{\substack{d<D_1\\ d|P(z)}}\lambda^+(d)
                 \sum_{M\leqslant n< y}\bigg(1-\frac{\log n}{\log y}\bigg)
                 \frac{x(1+O(\varepsilon))}{[d,n]\varphi(q)}
                       \nonumber \\
= & \,\, \frac{x(1+O(\varepsilon))}{\varphi(q)}\sum_{\substack{d<D_1\\ d|P(z)}}\frac{\lambda^+(d)}{d}
                 \sum_{M\leqslant n< y}\bigg(1-\frac{\log n}{\log y}\bigg)\frac{(d,n)}{n}.
\end{align}
For the inner sum in (\ref{sum_2-upper-2}), by partial summation, we have
\begin{align}\label{sum_2-upper-3}
  & \,\, \sum_{M\leqslant n< y}\bigg(1-\frac{\log n}{\log y}\bigg)\frac{(d,n)}{n}
                         \nonumber \\
= & \,\, \sum_{v|d}\sum_{\substack{\frac{M}{v}\leqslant n_1<\frac{y}{v}\\ (n_1,\frac{d}{v})=1}}
         \bigg(1-\frac{\log(n_1v)}{\log y}\bigg)\frac{1}{n_1}
                         \nonumber \\
= & \,\,\sum_{v|d}\sum_{\frac{M}{v}\leqslant n_1<\frac{y}{v}}\bigg(1-\frac{\log(n_1v)}{\log y}\bigg)\frac{1}{n_1}
        \sum_{\alpha\mid(n_1,\frac{d}{v})}\mu(\alpha)
                          \nonumber \\
= & \,\, \sum_{v|d}\sum_{\alpha\mid\frac{d}{v}}\frac{\mu(\alpha)}{\alpha}
         \sum_{\frac{M}{v\alpha}\leqslant n_2<\frac{y}{v\alpha}}
         \bigg(1-\frac{\log(n_2 v\alpha)}{\log y}\bigg)\frac{1}{n_2}
                          \nonumber \\
= & \,\, \sum_{v|d}\sum_{\alpha\mid\frac{d}{v}}\frac{\mu(\alpha)}{\alpha}
         \int_{\frac{M}{v\alpha}}^{\frac{y}{v\alpha}}\bigg(1-\frac{\log(tv\alpha)}{\log y}\bigg)
         \frac{\mathrm{d}t}{t}+O\Bigg(\frac{1}{M}\sum_{v|d}v\sum_{\alpha\mid\frac{d}{v}}1\Bigg)
                           \nonumber \\
= & \,\, \sum_{v|d}\sum_{\alpha\mid\frac{d}{v}}\frac{\mu(\alpha)}{\alpha}\int_M^y
         \bigg(1-\frac{\log t}{\log y}\bigg)\frac{\mathrm{d}t}{t}
         +O\Bigg(\frac{1}{M}\sum_{v|d}v\tau\bigg(\frac{d}{v}\bigg)\Bigg).
\end{align}
For the integral in (\ref{sum_2-upper-3}), it is easy to see that
\begin{equation}\label{sum_2-upper-3-integral}
  \int_M^y\bigg(1-\frac{\log t}{\log y}\bigg)\frac{\mathrm{d}t}{t}
  =\frac{1}{2\log y}\bigg(\log\frac{y}{M}\bigg)^2.
\end{equation}
In addition, we have
\begin{align}\label{sum_2-upper-3-sum}
  \omega_1(d):=\sum_{v|d}\sum_{\alpha|\frac{d}{v}}\frac{\mu(\alpha)}{\alpha}
              =\sum_{\alpha|d}\frac{\mu(\alpha)}{\alpha}\sum_{v\mid\frac{d}{\alpha}}1
              =\sum_{\alpha|d}\frac{\mu(\alpha)}{\alpha}\tau\bigg(\frac{d}{\alpha}\bigg)
              =\prod_{p|d}\bigg(2-\frac{1}{p}\bigg).
\end{align}
Combining (\ref{sum_2-upper-1})--(\ref{sum_2-upper-3-sum}), we derive that
\begin{align}\label{sum_2-upper-twice}
 & \,\,\sum_{\substack{M\leqslant p<y\\ p\nmid q}}\bigg(1-\frac{\log p}{\log y}\bigg)S(\mathscr{A}_p,z)
                      \nonumber \\
 \leqslant & \,\, \frac{x(1+O(\varepsilon))}{\varphi(q)}\sum_{\substack{d<D_1\\ d|P(z)}}\frac{\lambda^+(d)}{d}
                  \Bigg(\frac{\omega_1(d)}{2\log y}\bigg(\log\frac{y}{M}\bigg)^2
                  +O\bigg(\frac{1}{M}\sum_{v|d}v\tau\bigg(\frac{d}{v}\bigg)\bigg)\Bigg)
                        \nonumber \\
 = & \,\, \frac{x(1+O(\varepsilon))}{2\varphi(q)\log y}\bigg(\log\frac{y}{M}\bigg)^2
          \sum_{\substack{d<D_1\\ d|P(z)}}\frac{\omega_1(d)}{d}\lambda^+(d)
          +O\Bigg(\frac{x}{\varphi(q)M}
             \sum_{\substack{d<D_1\\ d|P(z)}}|\lambda^+(d)|\sum_{v|d}\frac{\tau(d/v)}{d/v}\Bigg)
                       \nonumber \\
 = & \,\, \frac{x(1+O(\varepsilon))}{2\varphi(q)\log y}\bigg(\log\frac{y}{M}\bigg)^2
          \sum_{\substack{d<D_1\\ d|P(z)}}\frac{\omega_1(d)}{d}\lambda^+(d)
          +O\Bigg(\frac{x}{\varphi(q)M}\sum_{\substack{d<D_1\\ d|P(z)}}\sum_{v|d}\frac{\tau(v)}{v}\Bigg).
\end{align}
By noting that the function $\omega_1(d)$ is multiplicative and it satisfies the $2$--dimensional sieve assumptions, we specify $\lambda^+(d)$'s to be that from Selberg's $\Lambda^2$--sieve and deduce that (for instance, one can see p.197 of \cite{Halberstam-Richert-book})
\begin{equation}\label{Selberg-Lambda^2-sieve}
\sum_{\substack{d<D_1\\ d|P(z)}}\frac{\omega_1(d)}{d}\lambda^+(d)=\frac{1}{G(D_1,z)}
=\frac{\mathscr{V}(z)}{\sigma(s)}\bigg(1+O\bigg(\frac{1}{\log z}\bigg)\bigg)
\end{equation}
holds for $z\leqslant D_1$, where
\begin{equation}\label{Selberg-Lambda^2-condition}
 s=\frac{\log D_1}{\log z},\quad \mathscr{V}(z)=\prod_{p<z}\bigg(1-\frac{\omega_1(p)}{p}\bigg),\quad
 \sigma(s)=\frac{s^2}{8e^{2\gamma}}\quad \textrm{for}\quad 0<s\leqslant2.
\end{equation}
By (\ref{sum_2-upper-3-sum}) and Mertens' prime number theorem (See \cite{Mertens-1874}), we obtain
\begin{equation}\label{Selberg-Lambda^2-V(z)-explicite}
 \mathscr{V}(z)=\prod_{p<z}\bigg(1-\frac{1}{p}\bigg)^2=\frac{e^{-2\gamma}}{\log^2z}
 \bigg(1+O\bigg(\frac{1}{\log z}\bigg)\bigg).
\end{equation}
Taking $D_1=N^2$, then $z\leqslant D_1$, and thus (\ref{Selberg-Lambda^2-sieve}) holds. Moreover, the remainder term in (\ref{sum_2-upper-twice}) is
\begin{align}\label{error-contribution}
  \ll &\,\,  \frac{x}{\varphi(q)M}\sum_{d<D_1}\prod_{p|d}\bigg(1+\frac{2}{p}\bigg)
            \ll \frac{x}{\varphi(q)M}\sum_{d<D_1}(\log\log d)^2
                          \nonumber \\
 \ll &\,\,    \frac{x}{\varphi(q)M}D_1(\log\log D_1)^2=o\bigg(\frac{x}{\varphi(q)\log y}\bigg).
\end{align}
From (\ref{sum_2-upper-twice})--(\ref{error-contribution}), we deduce that
\begin{align}\label{second-sum-last}
  \sum_{\substack{M\leqslant p<y\\ p\nmid q}}\bigg(1-\frac{\log p}{\log y}\bigg)S(\mathscr{A}_p,z)
      \leqslant & \,\,\frac{x(1+O(\varepsilon))}{\varphi(q)\log y}\bigg(\frac{\log(y/M)}{\log N}\bigg)^2
                       \nonumber \\
  = & \,\, \frac{x(1+O(\varepsilon))}{\varphi(q)\log D}\cdot\frac{\log D}{\log y}
             \bigg(\frac{\log(y/M)}{\log N}\bigg)^2
                       \nonumber \\
   = & \,\, \frac{x(1+O(\varepsilon))}{\varphi(q)\log D}
           \Bigg\{\frac{6\theta-7}{\delta}\bigg(\frac{2(\delta-\theta+1)}{2\theta-3}\bigg)^2\Bigg\}.
\end{align}
Finally, combining (\ref{aim-lower}), (\ref{W(A;z,y)-expan}), (\ref{S(A,z)-lower}), (\ref{W-second-sum-divided}), (\ref{sum_3-1}), (\ref{sum_3-coeff}), (\ref{sum_1-1}), (\ref{sum_1-coeff}) and (\ref{second-sum-last}), we conclude that
\begin{align*}
    & \,\, \sum_{\substack{n\in\mathscr{A}\\ (n,P(z))=1\\ \Omega(n)\leqslant 2}}\mathcal{W}(n)    
           \geqslant W(\mathscr{A};z,y)+o\bigg(\frac{x^{1-\varepsilon}}{\varphi(q)}\bigg)
                        \nonumber \\
 \geqslant &\,\, \frac{x(1+O(\varepsilon))}{\varphi(q)\log D}\Bigg\{2\bigg(\log\frac{18}{5}+
                 \int_3^\frac{18}{5}\frac{\mathrm{d}t_1}{t_1}\int_2^{t_1-1}\frac{\log(t_2-1)}{t_2}\mathrm{d}t_2\bigg)
                        \nonumber \\
         &\,\,-\frac{6\theta-7}{2(3\delta-\theta)}\int_{\frac{30\theta-35}{92}}^{\frac{12\theta-14}{23}}
               \frac{\delta-\beta}{\beta(\frac{3\theta}{2}-\frac{7}{4}-\beta)}
               \Bigg(1+\int_2^{\frac{108\theta-126-92\beta}{30\theta-35}}
               \frac{\log(t-1)}{t}\mathrm{d}t\Bigg)\mathrm{d}\beta
                     \nonumber \\
         &\,\, -\frac{6\theta-7}{2(3\delta-\theta)}\int_{\frac{12\theta-14}{23}}^{\theta-1}
                \frac{\delta-\beta}{\beta(\frac{3\theta}{2}-\frac{7}{4}-\beta)}\mathrm{d}\beta      
                      -\frac{6\theta-7}{3\delta-\theta}
                   \bigg(\frac{2(\delta-\theta+1)}{2\theta-3}\bigg)^2 \Bigg\}.
\end{align*}
By recalling the parameter $\delta=0.86$ and $\theta=1.8345$, then by a simple numerical calculation, we know that the number in the above brackets $\{\,\}$ is $>0.0004282583$. This completes the proof of Theorem \ref{Theorem}.

\section*{Acknowledgement}

% The authors would like to express the most sincere gratitude to the referee for his/her patience in refereeing this paper.

The authors would like to express the most sincere gratitude to the referee
for his/her patience in refereeing this paper. This work is supported by the National Natural Science Foundation of China
(Grant No. 11901566, 12001047, 11771333, 11971476, 12071238), the Fundamental Research Funds for the Central Universities
(Grant No. 2021YQLX02), the National Training Program of Innovation and Entrepreneurship for Undergraduates
(Grant No. 202107010), the Undergraduate Education and Teaching Reform and Research Project for China University of Mining 
and Technology (Beijing) (Grant No. J210703), and the Scientific Research Funds of Beijing Information Science and Technology University (Grant No. 2025035).

\end{document}